\documentclass[a4paper,12pt]{article}
\usepackage{amsmath,amsthm,amssymb,latexsym,epsfig,graphicx}
\usepackage{enumerate}

\newtheorem*{teor}{Theorem}

\newtheorem*{lemanom}{Lemma}

\theoremstyle{definition}

\newtheorem*{gracies}{Acknowledgements}

\newcommand{\ep}{\varepsilon}
\newcommand{\C}{\mathbb{C}}

\begin{document}

\title{{\bf The maximal Beurling transform associated with squares}}
\author{  Anna Bosch-Cam\'{o}s, Joan Mateu and  Joan Orobitg}

\date{}

\maketitle

\begin{abstract}

It is known that the improved Cotlar's inequality $B^{*}f(z) \le C M(Bf)(z)$, $z\in\C$,
holds for the Beurling transform $B$, the maximal Beurling transform
$B^{*}f(z)=$ $\displaystyle\sup_{\ep >0}\left|\int_{|w|>\ep}f(z-w) \frac{1}{w^2} \,dw\right|$, $z\in\C$,
and the Hardy--Littlewood maximal operator $M$. In this note we consider the maximal Beurling transform associated with
squares, namely, $B^{*}_Sf(z)=\displaystyle\sup_{\ep >0}\left|\int_{w\notin Q(0,\ep)}f(z-w) \frac{1}{w^2} \,dw \right|$, $z\in\C$,
$Q(0,\ep)$ being the square with sides parallel to the coordinate axis of side length $\ep$.
We prove that $B_{S}^{*}f(z) \le C M^2(Bf)(z)$, $z\in\C$, where $M^2=M \circ M$ is the iteration of the Hardy--Littlewood maximal operator,
and $M^2$ cannot be replaced by $M$.

\end{abstract}


\section{Introduction}

Although we only work with the Beurling transform and its iterations,
we begin by considering the question for more general Calder\'on-Zygmund  singular operators.
Let $T$ be a smooth homogeneous Calder\'on-Zygmund singular integral operator on~$\mathbb R^n$ with kernel 
\begin{equation} \label{nucli}
K(x)=\frac{\Omega(x)}{|x|^n}, \quad  x\in \mathbb R^n\setminus \{0\},
\end{equation}
where $\Omega$ is a homogeneous function of degree 0 whose restriction 
to the unit sphere $S^{n-1}$ is $C^{\infty}$ and satisfies the cancellation property 
$$
\int_{|x|=1}\Omega(x)\,d\sigma(x)=0, 
$$
 $\sigma$ being the normalized surface measure in $S^{n-1}$. Let $Tf$ be the principal value convolution operator 
\begin{equation} \label{def}
Tf(x)= \text{p.v.} \int f(x-y)K(y)\,dy\equiv\lim_{\varepsilon \rightarrow 0}T^{\ep}f(x), 
\end{equation}
where $T^{\ep}f$ is the truncated operator at level ${\ep}$ defined by 
\begin{equation} \label{truncated}
T^{\ep}f(x)=\int_{|y|>\ep}f(x-y)K(y)\,dy =\int_{|x-y|>\ep}f(y)K(x-y)\,dy. 
\end{equation}
For $f \in L^p, \, \, 1\le p < \infty$, the limit in \eqref{def} exists for almost all $x$. 
One says that the operator $T$ is even (or odd) if the kernel \eqref{nucli} is even (or odd), that is, if
$\Omega(-x)= \Omega(x)$, $x\in \mathbb R^n\setminus \{0\}$ (or $\Omega(-x)= -\Omega(x)$, $x\in \mathbb R^n\setminus \{0\}$).
Let $T^*$ be the maximal singular integral 
$$
T^*f(x)=\sup_{\ep >0}|T^{\ep}f(x)|, \quad  x\in \mathbb R^n.
$$
The classical Cotlar's inequality (e.g.\ \cite[p.~295]{Gr1}) asserts that
$$
T^*f(x)\le C(M(Tf)(x) + M(f)(x)),\quad x\in \mathbb R^n,
$$
where $M$ denotes the Hardy--Littlewood maximal operator. In  \cite{MV} one proved the pointwise inequality
\begin{equation}\label{B1}
 B^*f(z)\le M(Bf)(z), \quad z\in\mathbb{C},
\end{equation}
where $Bf$ is the Beurling transform of $f$ defined by
\begin{equation}\label{beurling}
 Bf(z)=  \text{p.v.} \frac{-1}{\pi}\int \frac{f(z-w)}{w^2} \, dw, \quad z\in\mathbb{C},
\end{equation}
$dw$ denoting the 2-dimensional Lebesgue measure.
The work \cite{MOV} addresses the issue of characterizing even smooth homogeneous Calder\'on--Zygmund operators 
for which the inequality \eqref{B1} is true. In particular, it shows that 
$T^*f(x)\le C M(Tf)(x)$, $x\in\mathbb R^n$, if and only if $\| T^*f\|_2 \le C \|Tf\|_2$, $f\in L^2(\mathbb R^n)$.
The analogous question for odd operators is treated in \cite{MOPV}, showing that
$T^*f(x)\le C M^2(Tf)(x)$, $x\in\mathbb R^n$, if and only if $\| T^*f\|_2 \le C \|Tf\|_2$, $f\in L^2(\mathbb R^n)$.
(See \cite{BMO} for an $L^p$ version of these results.)

In the definition \eqref{truncated} of the truncated operator $T^{\ep}f$  we are taking the kernel $K$ outside a ball of radius $\ep$.
We also could define a truncation outside a cube of sidelength $\ep$. Let $Q(x,\ep)$ be the cube, with sides parallel to axis, centered at
$x$ of side length $\ep$ and define
\begin{equation*} 
T^{\ep}_Qf(x)=\int_{y\notin Q(0,\ep)}f(x-y)K(y)\,dy =\int_{y\notin Q(x,\ep)}f(y)K(x-y)\,dy,
\end{equation*}
and consider the associated maximal singular integral 
$$
T^*_Sf(x)=\sup_{\ep >0}|T_Q^{\ep}f(x)|, \quad x\in \mathbb R^n.
$$
By simple geometry one checks that
\begin{align*}
T^{\ep}_Qf(x) &= T^{\sqrt{n}\ep/2}f(x)+ \int_{B(0,\sqrt{n}\ep/2)\setminus Q(0,\ep)} f(x-y)K(y)\,dy,\\
T^{\ep}f(x) &= T^{2\ep}_Qf(x)+ \int_{Q(0,2\ep)\setminus B(0,\ep)} f(x-y)K(y)\,dy,
\end{align*}
and therefore
\begin{align*}
 T^*_Sf(x) &\le T^*f(x) + C M(f)(x),\\
 T^*f(x) &\le T^*_Sf(x) + C M(f)(x).
\end{align*}
Consequently, from the $L^p$-theory point of view, the maximal operators $T^*$ and $T^*_S$ are equivalent. On the other hand,
we wonder if there is any pointwise inequality that relates $T^*_S f(x)$ with $Tf(x)$ like in \eqref{B1}. In this note we provide an answer
for the Beurling transform defined in \eqref{beurling} and for its $k$th iteration $B^k= B\circ \cdots \circ B$. The Fourier multiplier of $B$
is $\dfrac{\bar\xi}{\xi}$, or in other words, $\widehat{Bf}(\xi)=\dfrac{\bar\xi}{\xi} \hat f(\xi)$ and then $B$ is an isometry on $L^2(\C)$.
The kernel $b_k$ of $B^k$ may be computed explicitly,
for instance via a Fourier transform argument \cite[p.~73]{St}, and one obtains
$$
b_k(z)=\dfrac{(-1)^k k}{\pi } \dfrac{\bar z^{k-1}}{z^{k+1}}, \quad z\neq 0.
$$ 
Similarly we get that the kernel of the inverse operator, $(B^k)^{-1}$, is precisely the conjugate kernel:
$\dfrac{(-1)^k k}{\pi } \dfrac{ z^{k-1}}{\bar z^{k+1}}$.

\begin{teor}
\begin{enumerate}[(a)]
 \item If $k$ is odd, then
 $$
 (B^k)^*_S f(z)\le C M^2(B^kf)(z),\quad z\in\mathbb C,
  $$
  where $C$ is a constant depending on $k$ and $M^2=M \circ M$ the iterated maximal operator.
  \item If $k$ is even, for any positive constant $C$, for any $z\in\C$ and for any positive integer $j$ there exists a 
  function $f$ in $L^2(\mathbb{C})$ such that
  $$
 (B^k)^*_S f(z)\ge C M^j(B^kf)(z),
  $$
  where $M^j =M\circ \cdots \circ M$ is the $j$th iteration of $M$.
  \end{enumerate}
\end{teor}

\medskip

Notice that in the right hand side of \eqref{B1} there is the maximal operator, but in~(a) of the Theorem the iteration
of $M$ appears. The reason is that $B^{-1}(b_1\chi_{\C\setminus D(0,1)})$ is a bounded function with compact support, but 
$B^{-1}(b_1\chi_{\C\setminus Q(0,1)})$ is an unbounded BMO function without compact support.
We also provide an example to show that in part (a) of the above theorem for $k=1$ we cannot replace
the iterated Hardy--Littlewood maximal operator $M^2=M\circ M$ by $M$. We believe that similar examples should also give optimality 
for the case $k = 3,5,\dots$ . Property (b) of the Theorem is also satisfied for the Cauchy transform on a Lipschitz graph
when the graph is not a line (see \cite{Gi}).

We adhere to the usual convention of denoting by $C$ a positive constant, independent
of the relevant parameters involved, and which may vary from an occurrence to another.
Our notation and terminology are standard. For example, $A\approx B$ means that
the two quantities $A$ and $B$ satisfy the relation $C^{-1} A \le B \le CA$, for some constant
$C\ge 1$.

\section{Proof of the Theorem}
We proceed as in \cite{MOPV}. By translating and dilating one reduces the proof to
\begin{equation}\label{desigualtat0}
 |(B^k)^2_Q f(0)|\le C M^2(B^kf)(0),
\end{equation}
where $(B^k)^2_Q f(0)$ is the truncated integral at level 2. Denote the square $Q(0,2)=[-1,1]\times [-1,1]$ by $Q_0$.
As usual, $\chi_E$ denotes the characteristic function of the set~$E$.
The idea is to obtain an identity of the form
\begin{equation*}
 b_k(z)\chi_{\C\setminus Q_0}(z) = B^k(a_k)(z),
 \end{equation*}
for some function $a_k$. Since $B^k$ is an invertible operator, we have
\begin{equation*}
a_k = (B^k)^{-1} (b_k\, \chi_{\C\setminus Q_0})
\end{equation*}
and so $a_k\in \operatorname{BMO}(\C)$. We claim that if $k$ is odd we have the decay estimate
\begin{equation}\label{decay}
 |a_k(z)|\le \frac{C_k}{|z|^3}, \quad \text{if } |z|> 3.
\end{equation}
Before proving the claim we show how it yields (a) of the Theorem.
We argue as in \cite[p.~3675]{MOPV}. For any $f\in L^2(\C)$ we have
\begin{equation}\label{E1}
\begin{split}
(B^k)^2_Q f(0) =  & \int_{z\notin Q_0} f(z) b_k(-z)\, dz  = \int f(z)b_k(z)\chi_{\C\setminus Q_0}(z)\, dz\\
               = & \int f(z) B^k(a_k)(z) \, dz = \int B^k f(z) a_k(z)\, dz\\
               = & \int_{|z|<3} B^kf(z)(a_k(z)- (a_k)_{D(0,3)})\, dz\\
 & +  (a_k)_{D(0,3)} \int_{|z|<3} B^kf(z)\, dz +  \int_{|z|>3} B^kf(z) a_k(z)\, dz\\
 := & \; I + II + III,
\end{split}
\end{equation}
where, as usual,  $(a_k)_{D(0,3)}= |D(0,3)|^{-1}\int_{D(0,3)} a_k(z) \,dz$. 
To estimate the term $I$ we use the generalized H\"older's inequality and the pointwise equivalence
$M_{L(\log L)}f(x)\simeq M^2f(x)$ (\cite{P}) to get
$$
|I| \le C\| a_k \|_{\operatorname{BMO}} \| B^kf\|_{L(\log L), D(0,3)}\le C M^2(B^kf)(0).
$$
Clearly,
$$
|II |\le C M(B^kf)(0).
$$
Finally, from the decay of $a_k$ we obtain
$$
| III| \le C \int_{|z|>3} \frac{|B^kf(z)|}{|z|^{3}}\, dy\le C M(B^kf)(0),
$$
by using a standard argument which consists in estimating the integral
on the annuli $\{3^j \le |z| < 3^{j+1}  \}$. Therefore we get \eqref{desigualtat0}
and part (a) of the Theorem is proved. Now we turn to the proof of the claim.
We express $a_k$ as
\begin{equation*}
\begin{split}
 a_k & = (B^k)^{-1}(b_k\, \chi_{\C\setminus Q_0})=   (B^k)^{-1}( b_k - b_k\chi_{Q_0})\\
  & =  \delta_0 - (B^k)^{-1}( b_k\chi_{Q_0}),
\end{split}
\end{equation*}
where $\delta_0$ is the Dirac delta at the origin. We have to determine the decay of $(B^k)^{-1}( b_k\chi_{Q_0})(z)$
when $|z|>3$:
\begin{equation}\label{E2}
\begin{split}
\!\!\!\!(B^k)^{-1}( b_k\chi_{Q_0})(z)\!=& \lim_{\ep\to 0}\int_{\ep < |w|}\! \overline{b_k(z-w)} ( b_k\chi_{Q_0})(w)\, dw\\
=& \lim_{\ep\to 0}\int_{\ep < |w|}\! (\overline{b_k(z-w)} - \overline{b_k(z)})( b_k\chi_{Q_0})(w)\, dw \\
& \qquad  \qquad + \overline{b_k(z)}\lim_{\ep\to 0}\int_{\ep < |w|}\!  ( b_k\chi_{Q_0})(w)\, dw\\
=& \lim_{\ep\to 0}\int_{\ep < |w|}\! (\overline{b_k(z-w)} \!-\! \overline{b_k(z)})( b_k\chi_{Q_0})(w)\, dw \!+\! 
\overline{b_k(z)} B^k(\chi_{Q_0})(0)\\
:=&\, I_k +II_k .
\end{split}
\end{equation}
On the one hand, since 
$$
|\overline{b_k(z-w)} - \overline{b_k(z)}|\le \frac{C|w|}{|z|^{3}}, \quad |z|>3, \quad |w|\le \sqrt{2},
$$
we obtain
$$
|I_k|\le C \int_{w\in Q_0}\frac{|w|}{|z|^3}\, \frac{1}{|w|^2}\, dw = \frac{C}{|z|^3}.
$$
On the other hand, we will check in the next Lemma that when $k$ is odd $B^k(\chi_{Q_0})(0)\!=\!0$. Then $II_k=0$ and so we get \eqref{decay}.

\begin{lemanom}
 Let Q be any square, with sides parallel to axis, centered at $0$. Then
 \begin{enumerate}[(a)]
  \item $\overline{(B^k \chi_Q)(z)}= (B^k \chi_Q)(\bar z)$.
  \item  $(B^k \chi_Q)(i z)= (-1)^k (B^k \chi_Q)(z)$.
  \item $(B^k \chi_Q)(0)=0$ if $k$ is odd, and  $(B^k \chi_Q)(0)\neq 0$ if $k$ is even.
 \end{enumerate}
\end{lemanom}
\begin{proof}
 It is easy to check (a) and (b). Just consider the symmetry of the domain with respect to conjugation
 and with respect to rotation of angle $\pi/2$:
 \begin{equation*}
\begin{split}
 \overline{(B^k \chi_Q)(z)} &=  \text{p.v.}\int_{Q} \overline{b_k(z-w)}\, dw = \text{p.v.}\int_{Q} b_k(\bar z-\bar w)\, dw\\
                            &= \text{p.v.}\int_{Q} b_k(\bar z- w)\, dw = (B^k\chi_Q)(\bar z)
\end{split}
\end{equation*}
and
\begin{equation*}
\begin{split}
 (B^k \chi_Q)(i z) &= \text{p.v.}\int_{Q} b_k(i z- w)\, dw =  \text{p.v.}\int_{Q} b_k(i z- i w)\, dw \\
                            &= (-i)^{2k}\text{p.v.}\int_{Q} b_k( z- w)\, dw  = (-1)^k (B^k\chi_Q)(z).
\end{split}
\end{equation*}
From (b), when $k$ is odd, we have $(B^k \chi_Q)(0)=-(B^k \chi_Q)(0)$ and so  $(B^k \chi_Q)(0)=0$.
To get  $(B^k \chi_Q)(0)\neq 0$ when $k$ is even we have to work a little more. Set $k=2j$ and by simplicity let $Q$ be the square 
$Q_0$. Then
 \begin{equation*}
\begin{split}
(B^{2j} \chi_{Q_0})(0) &= \text{p.v.}\int_{Q_0} b_{2j}(w) \, dw = \int_{Q_0\setminus D(0,1)} b_{2j}(w) \, dw\\
    & = 4 \int_{E} b_{2j}(w) \, dw = 4\, \frac{2j}{\pi} \int_{E} \frac{\bar w^{2j-1}}{w^{2j+1}} \, dw,
\end{split}
\end{equation*}
where $E=\{w\in Q_0\setminus D(0,1): \; \operatorname{Re }w>0 \text{ and } \operatorname{Im} w >0 \}$. The second equality follows from 
the cancellation of the kernels and the third one from $b_{2j}(w)=b_{2j}(iw)$.
Therefore, making the change to polar coordinates $w=r e^{i\theta}$, with $1\le r\le \sqrt{2}$ and $\theta(r) \le \theta
\le \frac{\pi}{2} - \theta(r)$, $\theta(r) = \arccos(\frac{1}{r})$, we have
\begin{equation*}
\begin{split}
(B^{2j} \chi_{Q_0})(0) &= 4\frac{2j}{\pi} \int_{1}^{\sqrt{2}} \int_{\theta(r)}^{\frac{\pi}{2} - \theta(r)}
\frac{ (re^{-i\theta})^{2j-1}}{(re^{i\theta})^{2j+1}} \, d\theta \,r \, dr\\
    & = 4\frac{2j}{\pi} \int_{1}^{\sqrt{2}} \int_{\theta(r)}^{\frac{\pi}{2} - \theta(r)}
e^{-i\theta 4j} \, d\theta \, \frac{dr}{r}\\
     & = \frac{-4}{\pi} \int_{1}^{\sqrt{2}} \sin(\theta(r) 4j) \, \frac{dr}{r}.
\end{split}
\end{equation*}
Expressing $\sin(\theta(r) 4j)$ in terms of $\sin(\theta(r))$  and $\cos(\theta(r))$ and replacing 
$\sin(\theta(r)) = \frac{\sqrt{r^2 -1}}{r}$ and $\cos(\theta(r))=1/r$ we write
\begin{equation}\label{nozero}
\begin{split}
(B^{2j} \chi_{Q_0})(0) &= \frac{-4}{\pi}\sum_{m=0}^{2j-1} (-1)^m \binom{4j}{2m+1} 
\int_{1}^{\sqrt{2}} \frac{1}{r^{4j}} (r^2-1)^{\frac{2m+1}{2}} \, \frac{dr}{r}\\*[5pt]
&= \frac{-4}{\pi}\sum_{m=0}^{2j-1} (-1)^m \binom{4j}{2m+1} 
\int_{0}^{1}\frac{x^{2m+2}}{(x^2+1)^{2j+1}}\, dx \\*[5pt]
&:=  \frac{-4}{\pi} \int_{0}^{1} F_j(x) \, dx.
\end{split}
\end{equation}
Integrating by parts as many times as necessary, one easily obtains, for some $a,b\in \mathbb Q$,
$$
\int_{0}^{1}\frac{x^{2m+2}}{(x^2+1)^{2j+1}}\, dx = a + b \frac{\pi}{4} \neq 0.
$$
Although each factor in \eqref{nozero} is non-zero, there still might be a cancellation in the sum on $m$.
We will check that $(B^{2j} \chi_{Q_0})(0)\neq 0$ proving that a primitive of $F_j(x)$ is some rational function $R_j(x)$ with integer coefficients
 minus $\arctan(x)$. 
Thus we will have
$$
\int_0^1 F_j(x)\, dx= R_j(1)-R_j(0)-\frac{\pi}{4}\neq 0,
$$
because $R_j(1)-R_j(0)$ will be a rational number.

Given $n<d$, we define 
$$I(d,2n):=\int \frac{x^{2n}}{(x^2+1)^d}\, dx.
$$
It is not necessary to fix the constant of the primitive because we will value a
 definite integral. Integrating by parts we have
$$
I(d,2n)= \frac{-x^{2n-1}}{2(d-1) (x^2+1)^{d-1}} + \frac{2n-1}{2(d-1)}I(d-1,2(n-1)).
$$
Iterating the procedure we obtain
\begin{equation}\label{Ide1}
 I(d,2n)=  \frac{(2n-1)!}{2^{n-1}(n-1)!} \frac{(d-n-1)!}{2^n (d-1)!} I(d-n,0) + R(x),
\end{equation}
where $R(x)$ is some rational function defined on $\mathbb R$ and $R(0)=0$. 
Given $k\ge 2$ one easily gets
\begin{equation}\label{Ide2}
 I(k,0)=  \frac{(2k-3)!}{2^{2k-3}(k-1)!(k-2)!}  I(1,0) + R(x).
\end{equation}
Now, from \eqref{nozero} together with \eqref{Ide1} and \eqref{Ide2},
\begin{equation*}
\begin{split}
\int F_j(x)\, dx &= I(1,0)\left\{ \left(\sum_{m=0}^{2j-2}(-1)^m \binom{4j}{2m+1} \frac{(2m+1)!(4j-2m-3)!}{m!(2j)!(2j-m-2)!2^{4j-2}} \right) \right.\\
& \qquad\qquad\quad \left.\vphantom{\left(\sum_{m=0}^{2j-2}(-1)^m \binom{4j}{2m+1} \frac{(2m+1)!(4j-2m-3)!}{m!(2j)!(2j-m-2)!2^{4j-2}} \right)} - \frac{(4j)!}{(2j-1)!(2j)!2^{4j-1}}\right\} + R_j(x).
\end{split}
\end{equation*}
Performing some computations (see the Appendix), we get exactly what we wanted, 
\begin{equation}\label{suma}
\left(\sum_{m=0}^{2j-2}(-1)^m \binom{4j}{2m+1} \frac{(2m+1)!(4j-2m-3)!}{m!(2j)!(2j-m-2)!2^{4j-2}} \right)
 - \frac{(4j)!}{(2j-1)!(2j)!2^{4j-1}} = -1, 
\end{equation}
that is,
$$
\int F_j(x)\, dx = - I(1,0) + R_j(x)= -\arctan(x) +R_j(x).
$$
\end{proof}

Let us prove assertion (b) of the Theorem. Recall that now $k$ is even. By \eqref{E1} one has
$$
|(B^k)^2_Q f(0) | \lesssim M^j (B^kf)(0)
$$
if and only if 
\begin{equation}\label{E3}
 |III | \lesssim  M^j (B^kf)(0).
\end{equation}
By \eqref{E2}, when $|z|>3$,
$$
a_k(z)= \overline{b_k(z)}(B^k\chi_{Q_0})(0) + O\left(\frac{1}{|z|^3} \right):= \alpha_k\frac{z^{k-1}}{\bar z^{k+1}}+ O\left(\frac{1}{|z|^3} \right),
$$
where $\alpha_k$ is a non-zero constant that depends on $k$. Consequently, \eqref{E3} holds if and only if
$$
\left|\int_{|z|>3}\frac{z^{k-1}}{\bar z^{k+1}} B^kf(z)\, dz \right| \lesssim M^j(B^kf)(0), \quad f\in L^2.
$$
Since that $B^k$ is invertible in $L^2$, this is equivalent to
\begin{equation}\label{E4}
\left|\int_{|z|>3}\frac{z^{k-1}}{\bar z^{k+1}} G (z)\, dz \right| \lesssim M^j(G)(0), \quad f\in G^2. 
\end{equation}
But \eqref{E4} is false. Indeed, let $G$ be a function with compact support and $0 \le G \le 1$. Obviously, $M^j(G)\le 1$.
On the other hand, since $\frac{z^{k-1}}{\bar z^{k+1}}$ does not belong to $L^1(\C)$, we can make the left-hand side of 
\eqref{E4} as big as we want.

\section{Counterexample}

This section is devoted to prove that condition (a) in the Theorem is sharp for $k=1$. More precisely, we will prove that
 there exists a function $f$ such that for each constant $C>0$ there exists a point $z\in \mathbb C$ satisfying
\begin{equation}\label{E5}
 B^*_S f(z) > C M(Bf)(z).
 \end{equation}

We choose $f:= B^{-1}(\chi_{Q_0})$. For $|z|>2$ one has $M(Bf)(z)=M(\chi_{Q_0}))(z)\approx \frac{1}{|z|^2}$. 
So, in order to get inequality \eqref{E5} 
it is sufficient to prove that, for some $z$, one has
\begin{equation}\label{ma1}
 B^*_S f(z) \ge C \frac{\log|z|}{|z|^2}. 
\end{equation}

For $m>2$ (for instance, $m=5$), take $\alpha \gg m$, set $z=\alpha+ i\alpha$ and consider the truncated operator 
$$
(B) _Q^{2(\alpha+m)}f(z)=\frac{-1}{\pi}\int_{w\notin Q(z,2(\alpha+m))} \frac{f(w)}{(z-w)^2}\, dw.
$$
By definition $B^*_Q f(z)\ge |(B) _Q^{2(\alpha+m)}f(z)|$ and then we will have \eqref{ma1} if we prove
\begin{equation}\label{ma2}
 | (B) _Q^{2(\alpha+m)}f(z) | \gtrsim \frac{\log|z|}{|z|^2}. 
\end{equation}
The idea is to decompose $(B) _Q^{2(\alpha+m)}f(z)$ as a sum of certain terms. 
All terms, except one, can be bounded by $C|z|^{-2}$ and  the exceptional term will be of order $|z|^{-2}\log|z|$.
We begin by writing the equality
$$
(B) _Q^{2(\alpha+m)}f(z)=B^{\sqrt{2}(\alpha+m)}f(z)-\frac{1}{\pi}\int_{E}\frac{f(w)}{(w-z)^2}\, dw,
$$
where $E$ is the set $B(z,\sqrt{2}(\alpha+m)) \setminus Q(z,2(\alpha+m))$. From the pointwise inequality~\eqref{B1} the fisrt term
is bounded by $\frac{C}{|z|^2}$ and we just care about the second one. Set
$E= A_1\cup A_2\cup B$ where (see Figure 1)
\begin{equation*}
 \begin{split}
 A_{1}&:= \{w\in E: \,  \operatorname{Re}(w-z)<0,\,  \operatorname{Im}(w-z)<0 \text{ and}\, \operatorname{Im}(w+m+im)>0\},\\
A_2& := \{w\in E: \,  \operatorname{Re}(w-z)<0,\,  \operatorname{Im}(w-z)<0\text{ and}\, \operatorname{Im}(w+m+im)<0\},\\
B&:= E\setminus (A_1\cup A_2).
 \end{split}
\end{equation*}

\begin{center}
\includegraphics[width=8cm]{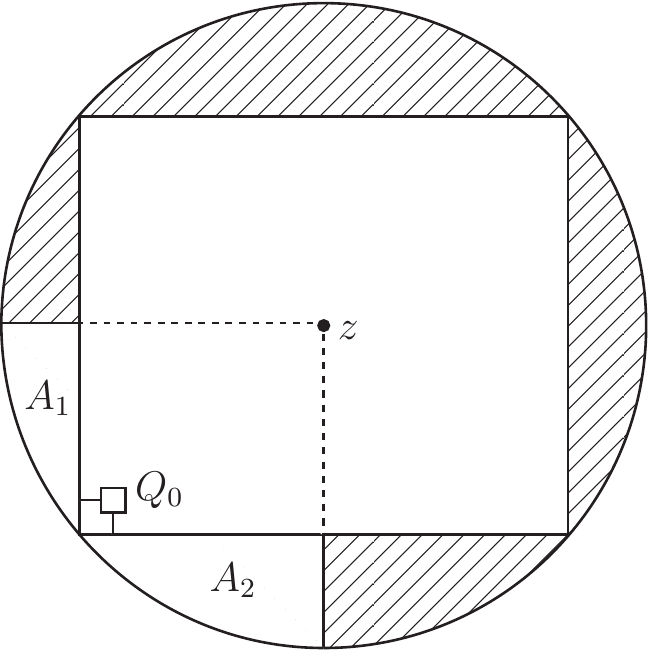}\\[2mm]
Figure 1
\end{center}

Thus
$$
\int_{E}\frac{f(w)}{(w-z)^2}\, dw= \int_{A_1\bigcup A_2}\frac{f(w)}{(w-z)^2}\, dw
+\int_{B}\frac{f(w)}{(\xi-z)^2}\, dw. 
$$
On the set $B$ we have $|f(w)|= |B^{-1}(\chi_{Q_0})(w)| \le \frac{C}{|w|^2}$ and then
$$
\left|\int_B \frac{f(w)}{(w-z)^2}\, dw \right|\le \int_B\frac{C}{|w -z|^2|w|^2} \, dw\le \frac{C|B|}{|z|^4}= \frac{C}{|z|^2} .
$$
 So, to get \eqref{ma2} it remains to prove that 
\begin{equation}\label{E6}
\left|\int_{A_1\cup A_2}\frac{f(w)}{(w-z)^2}\, dw\right| \ge C \frac{\log|z|}{|z|^2}.
\end{equation}
For any $w\in A_1 \cup A_2$ we write 
$$
f(w)= \frac{-1}{\pi}\int_{Q_0}\frac{1}{(\overline{w -\xi})^2}\, d\xi=
\frac{-1}{\pi}\frac{|Q_o|}{\overline{w}^2}-\frac{1}{\pi}\int_{Q_0} \left(\frac{1}{(\overline{w-\xi})^2}-\frac{1}{\overline w^2}
\right)\, d\xi.
$$
By the mean value property, the last integral in the above equality is bounded by~$\frac{C}{|w|^3}$. 
Therefore, an easy computation gives 
$$
\left|\int_{A_1\cup A_2} \left(f(w)+\frac{1}{\pi}\frac{|Q_0|}{\bar w^2}\right)\frac{1}{(w-z)^2}\, dw \right|
\le C \int_{A_1 \cup A_2}\frac{1}{|w|^3}\frac{1}{|w-z|^2}\, dw \le \frac{C}{|z|^2}. 
$$
Consequently, inequality \eqref{E6} holds if and only if 
\begin{equation}\label{E7}
 \left|\int_{A_1\cup A_2} \frac{1}{(z-w)^2}\frac{1}{\overline w^2}\, dw\right| \ge C \frac{\log|z|}{|z|^2} .
\end{equation}
Notice that, by the symmetry of the sets, $w\in A_1$ if and only if $i\overline{w}\in A_ 2$. 
Thus the integral in \eqref{E7} can be written as 
$$
\int_{A_1}\left( \frac{1}{(z-w)^2}\frac{1}{\overline{w}^2}+\frac{1}{(z-i\overline w)^2}\frac{1}{(iw)^2}\right) \, dw.
$$
Set $A_1^-=\{w\in A_1: \, \operatorname{Im}(w)<0\}$,  $A_1^+=\{w\in A_1: \bar w\in A_1^{-} \}$ and let $D=A_1\setminus (A_1^+\cup A_1^-)$ 
(see Figure 2).

\begin{center}
\includegraphics[width=8cm]{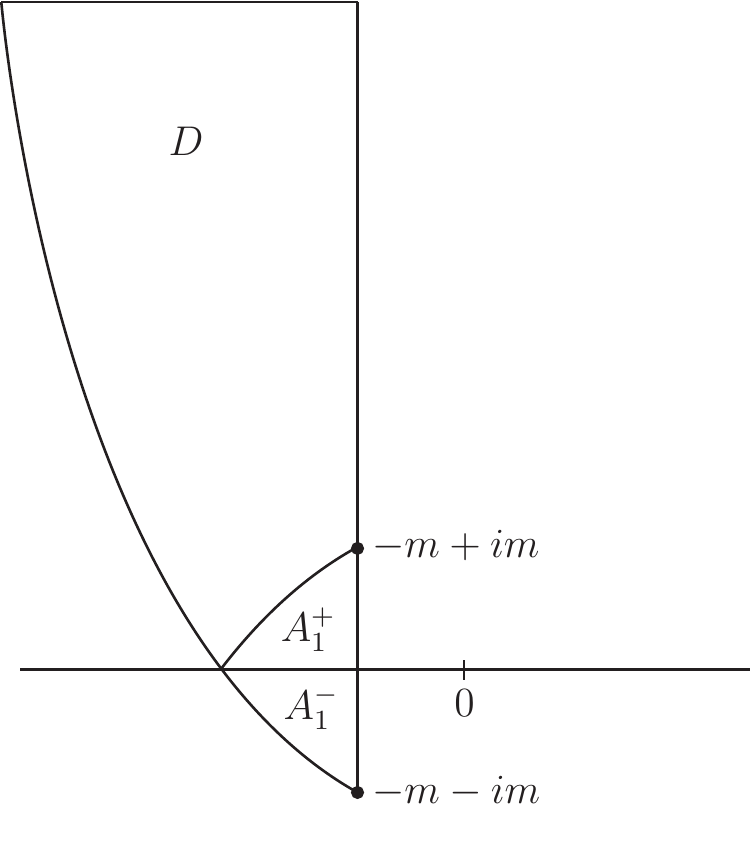}

Figure 2
\end{center}

First of all we prove that the integral over $A_1^+\cup A_1^-$ is bounded by $\frac{C}{|z|^2}$. Using a change of variable in the integral 
over $A_1^-$, one can see that the integral over $A_1^+ \cup A_1^-$ is
$$
\int_{A_1^+}\!\left(\frac{1}{(z-\overline w)^2}-\frac{1}{(z-i\overline w)^2}\right) \frac{1}{w^2}\, dw +
\int_{A_1^+}\!\left(\frac{1}{(z-w)^2}-\frac{1}{(z-iw)^2}\right) \frac{1}{\overline{w}^2}\, dw
:= I+II. 
$$
We are going to estimate $|I|$ and $|II|$ by $\frac{C}{|z|^2}$. Indeed,
\begin{equation*}
|I|=\left|\int_{A_1^+}\frac{-2\overline{w}^2+\overline{w}(2z-2iz)}{w^2(z-\overline{w})^2(z-i\overline{w})^2}\, dw\right|
\le \frac{C}{|z|^4}\int_{A_1^+}\left(1+\frac{|z|}{|w|}\right)\,dw \le \frac{C m}{|z|^3}\le \frac{C}{|z|^2}, 
\end{equation*}
where in the first inequality we have used that $|z-\overline{w}|=|z-i\overline{w}|\approx |z|$, and 
in the second one that $\int_{A_1^+}|w|^{-1}\le Cm$.
Using similar computations for the second term  we get 
$$
|II|= \left|\int_{A_1^+}\frac{-2w^2+ w(2z-2iz)}{\bar w^2(z-w)^2(z-i w)^2}\, dw\right|\le \frac{C}{|z|^2}.
$$
So, only remains to compute the integral over $D$, which we split again into two terms,
\begin{equation*}
 \begin{split}
\int_{D}\left(\frac{1}{(z-w)^2}\frac{1}{\overline{w}^2}-\frac{1}{(z-i\overline{w})^2}\frac{1}{w^2}\right)\,dw&\,=
\int_{D}\frac{1}{(z-w)^2}\left(\frac{1}{\overline{w}^2}-\frac{1}{w^2}\right)\,dw \\[7pt]
+ \int_{D}\frac{1}{w^2}\left(\frac{1}{(z-w)^2}-\frac{1}{(z-i\overline{w})^2}\right)\, dw &:=III+IV.
\end{split}
\end{equation*}
We note that if $w\in D$ then $|z-\overline{w}|=|z-i\overline{w}|\approx |z|$ and $m\le |w|\le 3 |z|$.
We treat the term $|IV|$ as before. In fact, we have
\begin{equation*}
 \begin{split}
|IV|&= \int_{D}\frac{2z(w-i\overline{w}) - \overline{w}^2-w^2}{w^2 (z-w)^2(z-i\overline{w})^2} \, dw
\le \int_{D} \frac{2|w|^2+4|z||w|}{|w|^2|z-w|^2|z-i\overline{w}|^2}\, dw\\[7pt]
&\le \frac{C}{|z|^4}\int_{D}\left(1+\frac{|z|}{|w|}\right)\, dw
\le \frac{C|D|}{|z|^4} + \frac{C}{|z|^3}\int_{m}^{3|z|} 1 \,dr
\le \frac{C}{|z|^2}.
\end{split}
\end{equation*}

\begin{center}
\includegraphics[width=8cm]{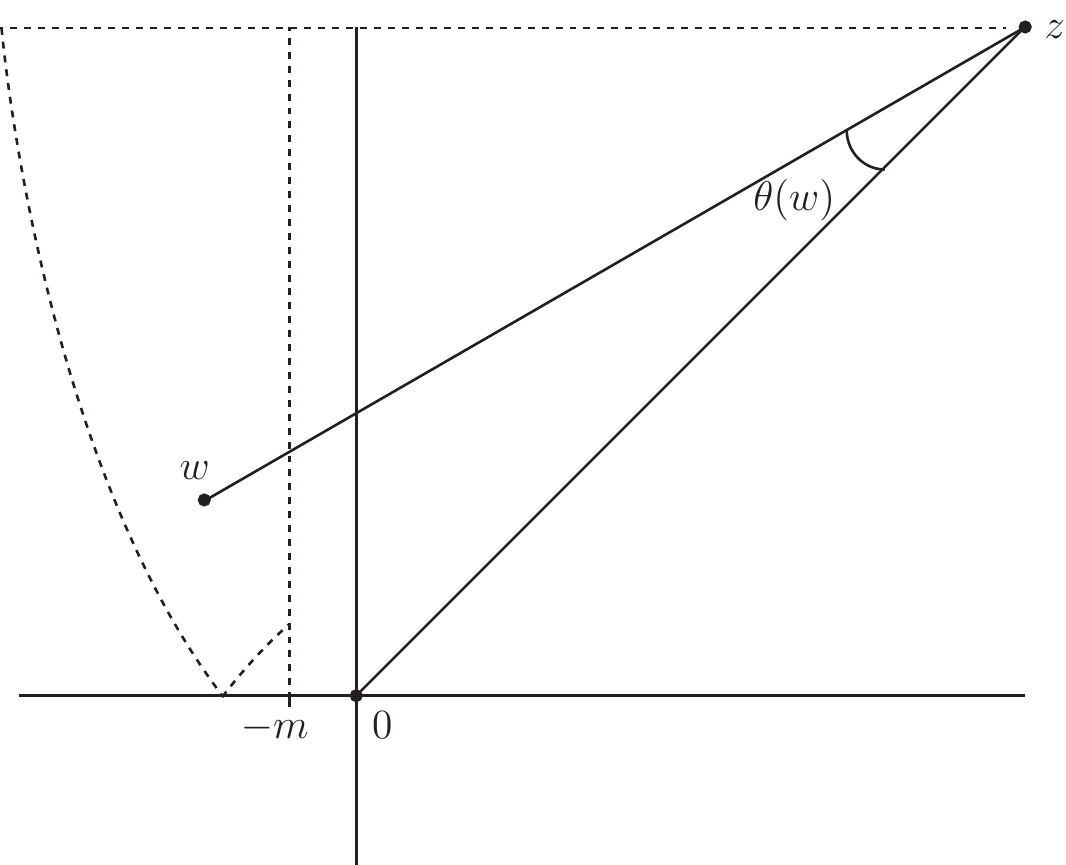}\\[2mm]
Figure 3
\end{center}

Finally the term $III$ will give us that our counterexample works.
For $w\in D$ (see Figure 3) we write $w-z=R(w)e^{i(\frac{5\pi}{4}-\theta(w))}$  where $R(w)=|w-z|\approx |z|$ and $0\le\theta(w)\le \frac{\pi}{4}$.
So, $i(w-z)^{-2}= -e^{i2\theta(w)} (R(w))^{-2}$. 
Then we have 
\begin{equation*}
 \begin{split}
  III&= 2i\int_{D}\frac{1}{(z-w)^2}\frac{\operatorname{Im} w^2}{|w|^4}\, dw 
  = -2\int_{D}\frac{e^{i2\theta(w)}}{R^2(w)} \frac{\operatorname{Im} w^2}{|w|^4}\, dw\\[2mm]
&=-2\int_{D}\frac{(\cos2\theta(w)+i\sin2\theta(w))}{R^2(w)}\frac{\operatorname{Im} w^2}{|w|^4}\, dw.
\end{split}
\end{equation*}

Since $\cos2\theta(w)\ge 0$ and $\operatorname{Im} w^2 \le 0$,
\begin{equation*}
 \begin{split}
  |\operatorname{Re} III|= 2\int_{D}\frac{\cos2\theta(w)}{R^2(w)}\frac{|\operatorname{Im} w^2|}{|w|^4}\, dw \ge
\frac{C}{|z|^2}\int_{D}\cos2\theta(w)\frac{|\operatorname{Im} w^2|}{|w|^4}\, dw.
 \end{split}
\end{equation*}
Fix $\delta >0$ and set
$$
\widetilde{D} :=\{w\in D\, : \, \cos2\theta(w)>\delta\}\cap\left\{w\in\C: \frac{\pi}{2}+\frac{\pi}{100}\le \operatorname{arg} w
\le \pi - \frac{\pi}{100}\right\}.
$$
Then, if $w\in \widetilde{D}$ we have $|\operatorname{Im} w|\ge |w| \frac{\beta}{\sqrt{1+\beta^2}}\ge  \frac{\beta|w|}{2}$,
where $\beta =\tan\frac{\pi}{50}$.
Note that the Lebesgue measure of $\widetilde{D}$ is comparable to the Lebesgue measure of $D$. Finally, using polar coordinates we get 
$$
|\operatorname{Re} III|\ge C\frac{\delta \beta}{|z|^2}\int_{\widetilde{D}}\frac{1}{|w|^2}\,dw\ge
C\frac{\delta \beta}{|z|^2}\int_{4m}^{|z|/10}\frac{dr}{r}\ge
C \frac{\delta \beta}{|z|^2}\log{|z|},
$$
which yields \eqref{E7} and the counterexample. 
 
\section{Appendix}
 In this section  we will prove, for the sake of the reader's convenience, identity \eqref{suma}. 
 To get it we will write the term in the left hand side of \eqref{suma} in another way. That is, 
\begin{equation*}
\begin{split}
&\left(\sum_{m=0}^{2j-2}(-1)^m \binom{4j}{2m+1} \frac{(2m+1)!(4j-2m-3)!}{m!(2j)!(2j-m-2)!2^{4j-2}} \right)
 - \frac{(4j)!}{(2j-1)!(2j)!2^{4j-1}}\\*[5pt]
&\hspace{3.5cm}=\frac{(4j)!}{(2j)!(2j-1)!2^{4j-1}}\sum_{m=0}^{2j-1}(-1)^m\binom{2j-1}{m}\frac{1}{4j-2m-1}\\*[5pt]
&\hspace{3.5cm}:= \frac{(4j)!}{(2j)!(2j-1)!2^{4j-1}} \; S,
\end{split}
\end{equation*}
where the last identity defines $S$.
Thus, only remains to prove that 
\begin{equation}\label{E9}
 S=-\frac{(4j)!}{(2j)!(2j-1)!2^{4j-1}}.
\end{equation}
Using twice the trivial fact that
$$
\sum_{m=0}^{2j-1}(-1)^m\binom{2j-1}{m}= (1-1)^{2j-1}=0,
$$
we have 
\begin{equation*}
 \begin{split}
 S&=\sum_{m=0}^{2j-1}(-1)^m\binom{2j-1}{m}\left(\frac{1}{ 4j-2m-1}-1\right)\\
&=(-1)2(2j-1)\sum_{m=0}^{2j-2}(-1)^m\binom{2j-1}{m}\frac{1}{4j-2m-1}\\
&=(-1)2(2j-1)\sum_{m=0}^{2j-2}(-1)^m\binom{2j-1}{m}\left(\frac{1}{4j-2m-1}-\frac{1}{3}\right)\\
&=(-1)^22^2(2j-1)(2j-2)\sum_{m=0}^{2j-3}(-1)^m\binom{2j-1}{m}\frac{1}{4j-2m-1}.
 \end{split}
\end{equation*}
Iterating this process $(2j-1)$ times we obtain \eqref{E9}, and so \eqref{suma} is proved.

\begin{gracies}
The authors were partially supported by grants numbers 2009SGR420
(AGAUR) and  MTM2010-15657 (MINECO).
\end{gracies}

\vspace{0.5 truecm}

\noindent
\begin{tabular}{@{}l}
Departament de Matem\`{a}tiques\\
Universitat Aut\`{o}noma de Barcelona\\
08193 Bellaterra, Barcelona, Catalonia\\\\

{\it E-mail:} {\tt annaboschcamos@gmail.com}\\ 
{\it E-mail:} {\tt mateu@mat.uab.cat}\\ 
{\it E-mail:} {\tt orobitg@mat.uab.cat}\\
\end{tabular}

\end{document}